\documentclass[11pt,reqno]{amsart}
\usepackage{mathtools}
\usepackage{comment}
\usepackage{xcolor}
\usepackage[a4paper, left = 2cm, right = 2cm, top = 1.5cm, bottom = 1.5cm]{geometry}
\usepackage{dsfont} 
\usepackage[all]{xy}
\usepackage{graphicx} 
\usepackage[
hidelinks]{hyperref}

\usepackage{tikz}  
\usetikzlibrary{shapes.geometric}

\usepackage{mdframed}

\DeclareMathAlphabet{\mathbbold}{U}{bbold}{m}{n}

\newtheorem{theorem}{Theorem}[section]
\newtheorem{proposition}[theorem]{Proposition}
\newtheorem{lemma}[theorem]{Lemma}

\theoremstyle{definition}
\newtheorem{definition}{Definition}[section]
\newtheorem{remark}[theorem]{Remark}

\newtheorem{question}{Question}

\newcommand{\cst}{\ensuremath{\mathrm{C}^\ast}}

\newcommand{\id}{\mathrm{id}}

\newcommand{\I}{\mathds{1}}

\newcommand{\is}[2]{{\left\langle{#1}\,\vline\,#2\right\rangle}}

\newcommand{\ket}[1]{{\left|#1\right\rangle}}
\newcommand{\bra}[1]{{\left\langle#1\right|}}

\newcommand{\CC}{\mathbb{C}}

\newcommand{\GG}{\mathbb{G}}

\DeclareMathOperator{\C}{C}

\DeclareMathOperator{\Mat}{\mathsf{Mat}}

\title{Quantum Mycielski Graphs}
\author{Arkadiusz Bochniak}
\address{Max Planck Institut f\"ur Quantenoptik, Hans-Kopfermann-Straße 1, 85748 Garching, Germany}
\address{Munich Center for Quantum Science and Technology, Schellingstraße 4, 80799 M{\"u}nchen, Germany}
\email{arkadiusz.bochniak@mpq.mpg.de}

\author{Pawe{\l} Kasprzak}
\address{Department of Mathematical Methods in Physics, Faculty of Physics, University of Warsaw, Poland}
\email{pawel.kasprzak@fuw.edu.pl}


\begin{document}
    \maketitle
\begin{abstract}
 {  The classical Mycielski transformation allows for constructing from a given graph the new one, with an arbitrarily large chromatic number but preserving the size of the largest clique contained in it. This particular construction and its specific generalizations were widely discussed in graph theory literature. Here we propose an analog of these transformations for quantum graphs and study how they affect the (quantum) chromatic number as well as clique numbers associated with them.}
\end{abstract}
\section{Introduction}

Properties and characteristics of classical graph colouring were widely studied from different perspectives, starting from their fundamental aspects, through applications in several branches of science and engineering, as well as in multiple daily life problems \cite{book}. The typical parameters one uses to characterize a given graph $G$ are related to a number of colors that can be used to label either its vertices or edges according to a certain set of rules. The most common characteristic is called the chromatic number $\chi(G)$ and is defined as the minimal number of colors that could be used to label the vertices of $G$ in such a way that none of the edges has the same colors associated to its two endpoints. The problem of determining this quantity for a generic graph is known to be NP-hard \cite{Karp1972}. Yet another NP-hard parameter  containing information about the structure of a given graph $G$ is its clique number $\omega(G)$ which determines the largest subgraph $K_n$ of $G$ whose every vertex shares an edge with any other vertex of this subgraph, the complete graph of $n$ vertices. In particular, if $\omega(G)=2$, then the graph is triangle-free (i.e. there is no closed loop formed out of three vertices). One could ask if starting from a triangle-free connected graph we can arbitrarily enlarge the chromatic number by adding a certain number of vertices and edges but at the same time, no triangle is generated as a subgraph and the resulting graph remains connected. The affirmative answer to this question was given by Mycielski \cite{Mycielski}. The resulting transformation that from a given graph $G$ produces a new graph $\mu(G)$ such that $G\subseteq \mu(G)$, $\omega(G)=\omega(\mu(G))$ and $\chi(\mu(G))=\chi(G)+1$ is referred to as the Mycielski transformation. This construction was later generalized in \cite{Stiebitz,VanNgoc}, and for it, the notion of generalized Mycielski transformation, Stiebitz transformation, or (higher) cones over a graph is used.

One area of intriguing applications of graph theory is the information theory \cite{Duan_2013,Stahlke}, where for a given noisy classical channel $\Phi_{\mathrm{cl}}: A\rightarrow \textrm{Prob}(B)
$ between two persons, Alice and Bob, one can associate the so-called confusability graph, defined as a complement of the distinguishability graph whose vertices are elements of the set $A$, and there exists an edge between two vertices $a_1,a_2\in A$ if and only if for all $b\in B$ we have $\Phi_{\mathrm{cl}}(b|a_1)\Phi_{\mathrm{cl}}(b|a_2)=0$, where $\Phi_{\mathrm{cl}}(b|a)$ stands for the conditional probability that Bob will receive $b\in B$ provided that Alice had sent $a\in A$. The fact that there is a connection between classical channels and graphs led to an intriguing possibility of associating to a quantum channel $\Phi_{\mathrm{q}}$ an object that will mimic the behavior of classical graphs - the quantum graph. This can be done either by directly using the associated Kraus operators defining the completely positive trace-preserving map $\Phi_{\mathrm{q}}$, or use them to construct certain operator systems (spaces) and study their properties. The concept of quantum graphs appeared already before in \cite{Erdos}, however, the relation to quantum information, in particular the problem of zero-error correction, significantly accelerated the interest in quantum graphs. Several formulations, which under some conditions turned out to be equivalent, were proposed. For a neat survey, we refer the reader to \cite{DawsQGT}. The two most widely used are:
\begin{enumerate}
    \item[I.] approach based on quantum relations and their formulations in terms of operator systems (spaces) \cite{weaver10,NW2015},
    \item[II.] formulation by mimicking in the quantum world the properties of adjacency matrices \cite{MRV2018,Brannan_2019}. 
\end{enumerate}

Let us again consider the aforementioned noisy classical channel $\Phi_{\mathrm{cl}}$, and now promote Alice and Bob to be players for the following game \cite{Stahlke}. Introducing an external referee, say, Charlie, one can consider the scenario in which Charlie, chooses a value $c$ from its own set $C$ and, with a conditional probability $P(x,u|c)$, sends a message $x\in X$ to Alice, and a value $u\in B$ to Bob. If Alice is equipped with a map $f: X\rightarrow A$, called strategy, she can use it to produce $a=f(x)$ and send this value to Bob using the channel $\Phi_{\mathrm{cl}}$ so that he will receive a value $b\in B$ with probability $\Phi_{\mathrm{cl}}(b|a)$. The goal of the players is to decode the initial value $c\in C$. With a set $X$ one can associate the so-called characteristic graph { defined in terms of $P$}, and it turns out \cite{Stahlke} that the existence of this type of scenario with a winning strategy is equivalent to the existence of a homomorphism between such characteristic graph and the distinguishability one associated with the channel $\Phi_{\mathrm{cl}}$. This observation suggests that certain characteristics of graphs could be potentially defined in terms of the existence of a winning strategy for certain types of games. In particular, this is true for the chromatic number. One can consider a two-player game in which Charlie chooses two vertices, say $v$ and $w$, of a given graph $G$ and sends one of them to Alice and the second one to Bob. They have to respond with numbers $\alpha,\beta$ from a fixed set $\{1,\ldots, c\}$ according to the following rules:
\begin{itemize}
    \item $v=w \Rightarrow \alpha=\beta$,
    \item $vw$ is an edge $\Rightarrow$ $\alpha\neq \beta$.
\end{itemize}
The lack of communication between Alice and Bob is assumed during the game, but they can agree on a strategy before the game started. It turns out that such a winning strategy exists if and only if $c\geq \chi(G)$. This motivated the concept of quantum chromatic number introduced in \cite{QchNb} where the players are no longer assumed to be classical but they are allowed to share some entangled quantum state $\Psi$ and, instead of providing answers, they are performing quantum measurements. More precisely, the players' strategy is mathematically modeled  by positive operator-valued measures (POVMs) $(E_{v\alpha})_{\alpha=1,\ldots, c}$ and $(F_{v\alpha})_{\alpha=1,\ldots, c}$, respectively, and it is a winning strategy if 
\begin{itemize}
    \item $\forall v\in V(G)\,  \forall \alpha\neq \beta \, \langle \Psi | E_{v\alpha}\otimes F_{v\beta}|\Psi \rangle =0$,
    \item $\forall vw\in E(G)$ $\, \forall\alpha \,\langle \Psi | E_{v\alpha}\otimes F_{w\alpha}|\Psi \rangle =0 $. 
\end{itemize}
The quantum chromatic number is defined, by analogy to its classical counterpart, as the minimal number for which there exists a winning strategy for this non-local game.

Both classical and quantum chromatic numbers, together with the notion of (quantum) coloring, can be also defined in the framework of quantum graphs \cite{BGH,Ganesan21}. It makes then sense to ask about analogs of Mycielski theorem in the quantum setting. In this paper, we propose such a generalization of the Mycielski transformation for quantum graphs and study how it affects parameters like (quantum) chromatic numbers as well as (several versions of) clique numbers. In Section~\ref{sec:Mycielski} we describe in detail the Mycielski transformation for classical graphs and discuss its generalization into cones over a graph. In Section~\ref{sec:quantum_graph} we briefly recall the notion of a quantum graph and establish the notation widely used in the rest of the paper. We introduce the (generalized) Mycielski transformation for quantum graphs in Section~\ref{sec:qMycielski} and show that indeed the resulting object is a well-defined quantum graph. Next, we study in Section~\ref{sec:chromatic} how this transformation affects the (quantum) chromatic number. The impact on the (different versions of) clique numbers is discussed in Section~\ref{sec:clique}. {Finally, in Section~\ref{sec:open} we briefly comment on other graph parameters and collect a series of open questions as well as potential future research directions.}

\section{Mycielski transformation for classical graphs}
\label{sec:Mycielski}
Let $G$ be an undirected graph with a given set of vertices $V(G)$ and the set of edges $E(G)$. The Mycielski transformation \cite{Mycielski} of $G$, denoted by $\mu(G)$, is the graph with the set of vertices $V(\mu(G))=\{\bullet\}\sqcup V(G)\sqcup V(G)$ s.t. every vertex from the second copy of $V(G)$ is connected with the distinguished vertex $\bullet$, and if $\{v_i,v_j\}\in E(G)$ and $u_i$ denotes the copy of $v_i$ is the second summand, then both $\{v_i,u_j\}$ and $\{v_j,u_i\}$ are edges in $\mu(G)$. In other words, the adjacency matrix $A_{\mu(G)}$ is of the form
\begin{equation}
    A_{\mu(G)}=\begin{pmatrix}
        0& \vec{0}^T & \vec{1}^T \\
        \vec{0} & A_G  & A_G \\
        \vec{1} & A_G  &\mathbbold{0}
    \end{pmatrix}.
\end{equation}
The Mycielski transformation $\mu(G)$ is a special case of so-called generalized Mycielski transformation $\mu_{r}(G)$ (or $r$-Mycielskian) for $r=1$, also known as cones over a graph $G$, \cite{Stiebitz,VanNgoc}. If the vertex set of $G$ is $V^0=\{v_1^0,v_2^0,\ldots,v_n^0\}$ and the edge set of $G$ is denoted by $E_0$ then the vertex set of $r$-Mycielskian is 
\[V^0\sqcup V^1\sqcup\ldots \sqcup V^r\sqcup\{\bullet\}\]
where $V^i =\{v_1^i,v_2^i,\ldots,v_n^i\}$ is the distinct copy of $V^0$ and the edge set is given by 
\[E = E_0\cup\bigcup_{i=0}^{r-1}\{v^i_jv^{i+1}_{j'}:v^0_jv^{0}_{j'}\in E_0\}\cup\{v^r_1\bullet,v^r_2\bullet,\ldots,v^r_n\bullet\}.\]

We illustrate this construction in a particular example. Let $G=K_2$, i.e. 
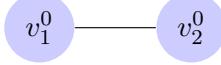
\begin{center}
\begin{tikzpicture}  
  [scale=1,auto=center,every node/.style={circle,fill=blue!20}]
    
  \node (a1) at (0,0) {$v_1^0$};  
  \node (a2) at (2,0) {$v_2^0$};  
  
  \draw (a1) -- (a2); 
  
\end{tikzpicture}  
\end{center}
The Mycielski transformation $\mu_1(G)$ of this graph is 
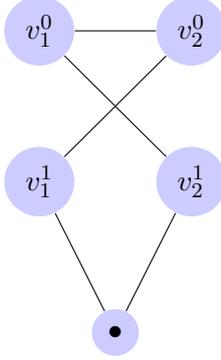
\begin{center}
\begin{tikzpicture}  
  [scale=1,auto=center,every node/.style={circle,fill=blue!20}]
    
  \node (a1) at (0,0) {$v_1^0$};  
  \node (a2) at (2,0) {$v_2^0$};
  \node (a3) at (0,-2) {$v_1^1$};
  \node (a4) at (2,-2) {$v_2^1$};
  \node (a5) at (1,-4) {$\bullet$};
  
  \draw (a1) -- (a2); 
  \draw (a1) -- (a4);
  \draw (a2) -- (a3);
  \draw (a3) -- (a5);
  \draw (a4) -- (a5);
  
\end{tikzpicture}  
\end{center}

Notice that this graph is simply
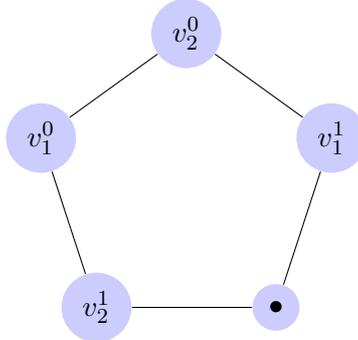
\begin{center}
\begin{tikzpicture}  

  \node[
      regular polygon,
      regular polygon sides=5,
      minimum width=40mm,
      ] 
      (PG) {}
      [scale=1,auto=center,every node/.style={circle,fill=blue!20}]
      (PG.corner 1) node (PG1) {$v_2^0$}
      (PG.corner 2) node (PG2) {$v_1^0$}
      (PG.corner 3) node (PG3) {$v_2^1$}
      (PG.corner 4) node (PG4) {$\bullet$}
      (PG.corner 5) node (PG5) {$v_1^1$}
    ;
    \foreach \S/\E in {
      1/2, 2/3, 3/4, 4/5, 5/1} 
      {
      \draw[-] (PG\S) -- (PG\E);
      }
\end{tikzpicture}  
\end{center}

Next, we observe that $\mu_2(G)$ is of the form
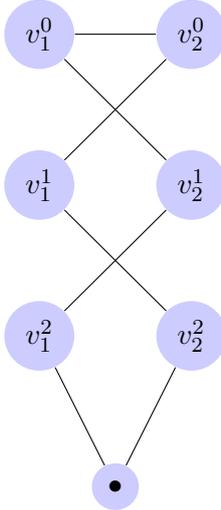
\begin{center}
\begin{tikzpicture}  
  [scale=1,auto=center,every node/.style={circle,fill=blue!20}]
    
  \node (a1) at (0,0) {$v_1^0$};  
  \node (a2) at (2,0) {$v_2^0$};
  \node (a3) at (0,-2) {$v_1^1$};
  \node (a4) at (2,-2) {$v_2^1$};
  \node (a5) at (0,-4) {$v_1^2$};
  \node (a6) at (2,-4) {$v_2^2$};
  \node (a7) at (1,-6) {$\bullet$};
  
  \draw (a1) -- (a2); 
  \draw (a1) -- (a4);
  \draw (a2) -- (a3);
  \draw (a3) -- (a6);
  \draw (a4) -- (a5);
  \draw (a5) -- (a7);
  \draw (a6) -- (a7);
\end{tikzpicture}  
\end{center}
We also notice that $\mu_1(\mu_1(G))\neq \mu_2(G)$.

\section{Quantum graphs}\label{sec:quantum_graph}

Let us recall here the definition of a quantum graph \cite[Definition 2.4]{DawsQGT}. 

\begin{definition}
\label{def:QG}
A triple $\mathcal{G}=(\mathbb{G},\psi, A)$ is a quantum graph, if
\begin{itemize}
    \item $\mathbb{G}$ is a finite quantum space; the corresponding (finite-dimensional) $\cst$-algebra will be denoted by $\C(\mathcal{G})$;
    \item  $\psi:\C(\mathcal{G})\rightarrow \mathbb{C}$  is a faithful state; the GNS space will be denoted by $L^2(\mathcal{G})$; the  multiplication map  when viewed as a linear map $L^2(\mathcal{G})\otimes L^2(\mathcal{G})\to L^2(\mathcal{G})$ will be denoted by $m$; since $\GG$ is finite and $\psi$ is faithful, $\C(\mathcal{G})$ and $L^2(\mathcal{G})$ can be identified as vector spaces; the map $ \mathbb{C}\ni \lambda \longmapsto \lambda\I_{\C(\mathcal{G})} \in L^2(\mathcal{G})$ will be denoted by $\eta$; note incidentally that 
    $\eta^\ast(x) = \psi(x)$ for all $x\in L^2(\mathcal{G})$;
    \item $\psi$ is a $\delta$-form, i.e.  
 $mm^\ast=\delta^2\id_{L^2(\mathcal{G})}$;
    \item $A$, called quantum adjacency matrix, is a self-adjoint map $A:L^2(\mathcal{G})\to L^2(\mathcal{G})$ s.t.
    \begin{align}\label{eq1:defA}  A&= \delta^{-2}m(A\otimes A)m^\ast \\\label{eq2:defA}
         A&=(\id\otimes \eta^\ast m)(\I\otimes A\otimes \I)(m^\ast \eta \otimes\id)  
    \end{align}
    
\end{itemize}
If, moreover, $m(A\otimes \I)m^\ast=\delta^2 \I$, then the quantum graph  $\mathcal{G}$ is called reflexive. On the other hand, if $m(A\otimes \I)m^\ast=0$ then the quantum graph is called irreflexive.
\end{definition}

We will often deal with a multiple of quantum graphs and, to avoid confusion, subscripts indicating the corresponding quantum graph will be added to the state, adjacency matrix, multiplication map, etc.~and e.g.~we shall write 
$\mathcal{G} = (\mathbb{G}_{\mathcal{G}},\psi_{\mathcal{G}}, A_{\mathcal{G}})$; $\dim\C(\mathcal{G})$ will be denoted by $|\mathcal{G}|$. 

Notice that
\begin{lemma}
\label{lem:eta}
    For a quantum graph $\mathcal{G}=(\mathbb{G},\psi, A)$ we have $(\id \otimes \eta^\ast)m^\ast=\id =(\eta^\ast \otimes\id )m^\ast $. 
\end{lemma}
\begin{proof}
 The lemma is proved by conjugating the identities  
 \[m(\eta\otimes\id) = \id = m(\id\otimes\eta),\] where the latter expresses the identity $a\I = a = \I a$ for all $a\in \C(\mathcal{G})$. 
\end{proof}
In what follows we shall  use Swedler  notation for $m^\ast :L^2(\mathcal{G})\longmapsto L^2(\mathcal{G})\otimes L^2(\mathcal{G})$ and write 
\[m^\ast (a) = a_{(1)}\otimes a_{(2)}.\]

The $\cst$-algebra $\C(\mathcal{G})$ will be viewed as a $\C^\ast$-algebra of operators acting on $L^2(\mathcal{G})$ and $\C(\mathcal{G})'$ will denote its commutant 
\[\C(\mathcal{G})' = \{T\in B(L^2(\mathcal{G})):Ta = aT \text{ for all } a\in\C(\mathcal{G})\}.\]
The map 
\begin{equation}
    P:B(L^2(\mathcal{G}))\ni X \longmapsto \delta^{-2}m(A\otimes X)m^\ast\in B(L^2(\mathcal{G}))
\end{equation} is a projection satisfying 
\[P(XT)=P(X)T,\quad P(TX) = TP(X),\quad P(X^\ast) = P(X)^\ast\] for all $X\in B(L^2(\mathcal{G}))$. In particular  the image $S$ of $P$ 
\begin{equation}\label{eq:opsys}\
    S = P(B(L^2(\mathcal{G})))   
\end{equation} is a selfadjoint operator subspace in $B(L^2(\mathcal{G}))$ which is also a bimodule over $\C(\mathcal{G})'$. Since $\C(\mathcal{G})$ is finite-dimensional, it is also a von Neumann algebra thus  the triple $(S,\C(\mathcal{G}),B(L^2(\mathcal{G})))$  is a quantum graph in the sense of Weaver \cite{weaver10,NW2015}. We shall usually write $S_{\mathcal{G}}$ for  the selfadjoint subspace $S\subset  B(L^2(\mathcal{G}))$ described above. Note that if $\mathcal{G}$ is reflexive then $S_{\mathcal{G}}$ is an operator space. The operator space $S_{G}$ for a classical graph $G$ is spanned by the matrix units $e_{ij}\in\Mat_{n}$ corresponding to edges $(ij)\in E(G)$, see Section~\ref{sec:clique}.

\section{Mycielski transformation for quantum graphs}
Prior to the definition of the (generalized) Mycielski transformation we need to fix a notation: given linear maps $T_i: V_i\to W$ for $i=1,2,\ldots,r$  the unique linear map $T:\bigoplus\limits_{i=1}^r V
_i\to W$ such that $T|_{V_i} = T_i$ for $i=1,2,\ldots,r$, will be denoted by $ \bigoplus\limits_{i=1}^r T_i$. 
\label{sec:qMycielski}
\begin{definition}
\label{def:MT}
The Mycielski transformation $\mu(\mathcal{G})$ of a quantum graph $\mathcal{G}=(\mathbb{G}_{\mathcal{G}},\psi_{\mathcal{G}},A_{\mathcal{G}})$ is a triple $(\mathbb{G}_{\mu(\mathcal{G})},\psi_{\mu(\mathcal{G})},A_{\mu(\mathcal{G})})$ where 
    \begin{itemize}
        \item $\mathbb{G}_{\mu(\mathcal{G})}=\bullet \sqcup \mathbb{G}_{\mathcal{G}}\sqcup \mathbb{G}_{\mathcal{G}}$ (i.e. the corresponding $\C^\ast$-algebra  is $\mathbb{C}\oplus \C(\mathcal{G}) \oplus \C(\mathcal{G}) $),
        \item $\psi_{\mu(\mathcal{G})}=\frac{1}{1+2\delta^2 }\left(\id \oplus\delta^2\psi_{\mathcal{G}}\oplus \delta^2 \psi_{\mathcal{G}}\right)$,
        \item $A_{\mu(\mathcal{G})}:L^2(\mu(\mathcal{G}))\to L^2(\mu(\mathcal{G}))$ is defined by 
    \begin{equation}\label{Amu} A_{\mu(\mathcal{G})}\begin{pmatrix}\lambda\\x\\y\end{pmatrix}=\begin{pmatrix}\delta^2\psi_\mathcal{G}(y)\\A_\mathcal{G}(x+y)\\\lambda\I_{\C(\mathcal{G})}+A_\mathcal{G}(x) \end{pmatrix} \end{equation} for all $\lambda\in\mathbb{C}$ and $x,y\in L^2(\mathcal{G})$. Note that we identify   $L^2(\mu(\mathcal{G}))$ with $\mathbb{C}\oplus L^2( \mathcal{G} ) \oplus L^2( \mathcal{G} )$, with the scalar product on the latter given by
    \[\is{\begin{pmatrix}
\lambda\\x\\y
\end{pmatrix}}{\begin{pmatrix}
\mu\\u\\v
\end{pmatrix}} = \frac{1}{2\delta^2+1}\left(\overline \lambda {\mu}+\delta^2\is{x}{u}+\delta^2\is{y}{v}\right).\]
\end{itemize}
\end{definition}

Instead of showing that this construction indeed defines a quantum graph, we first introduce the quantum analog of the generalized Mycielski graphs. We will show that the generalized Mycielski transformation for quantum graphs produces new quantum graphs. In particular, it will follow  that the Mycielski transformation as defined in Definition~\ref{def:MT} also has this property.

\begin{definition}\label{rmyciel}
    Let $\mathcal{G}=(\mathbb{G}_{\mathcal{G}},\psi_{\mathcal{G}},A_{\mathcal{G}})$ be a quantum graph   and let $r\geq 1$. The $r-1$-Mycielski transformation $\mu_{r-1}(\mathcal{G})$ of $\mathcal{G}$ is a triple $(\mathbb{G}_{\mu_{r-1}(\mathcal{G})},\psi_{\mu_{r-1}(\mathcal{G})},A_{\mu_{r-1}(\mathcal{G})})$ where 
    \begin{itemize}
        \item $\mathbb{G}_{\mu_{r-1}(\mathcal{G})}=\bullet \sqcup \underbrace{\mathbb{G}_{\mathcal{G}}\sqcup \ldots \sqcup \mathbb{G}_{\mathcal{G}}}_{r \text{ times}}$ (i.e. the   $\C^\ast$-algebra of $\mu_{r-1}(\mathcal{G})$  is $\mathbb{C}\oplus \C(\mathcal{G})^{\oplus r}$),
        \item $\psi_{\mu_{r-1}(\mathcal{G})}=\frac{1}{1+r\delta^2}\left(\id \oplus\delta^2\psi_{\mathcal{G}^{\oplus r}}\right)$,
        \item $A_{\mu_{r-1}(\mathcal{G})}:L^2(\mu_{r-1}(\mathcal{G}))\to L^2(\mu_{r-1}(\mathcal{G}))$ is defined by 
    \begin{equation}\label{Amur} A_{\mu_{r-1}(\mathcal{G})}\begin{pmatrix}\lambda\\x_1\\ \vdots \\ x_r \end{pmatrix}=\begin{pmatrix}\delta^2\psi_\mathcal{G}(x_r)\\A_\mathcal{G}(x_1+x_2)\\ A_{\mathcal{G}}(x_1+x_3)\\ \vdots \\ A_\mathcal{G}(x_k+x_{k+2})\\ \vdots \\ A_\mathcal{G}(x_{r-2}+x_r)\\ \lambda\I_\GG+A_\mathcal{G}(x_{r-1}) \end{pmatrix} \end{equation} for all $\lambda\in\mathbb{C}$ and $x_1,\ldots, x_r\in L^2(\mathcal{G})$. Here $L^2(\mu_{r-1}(\mathcal{G}))$ is identified with $\CC \oplus \bigoplus\limits_{k=1}^r L^2(\mathcal{G})$.
    \end{itemize}
\end{definition}
Let us right away observe that $\mu_1(\mathcal{G})$ as described in Definition \ref{rmyciel} coincides with Mycielskian $\mu(\mathcal{G})$ of $\mathcal{G}$ as defined in Definition \ref{def:MT}.
Note also that for every $r\geq 1$,  the scalar product on $L^2(\mu_{r-1}(\mathcal{G}))$ is given by
    \[
    \is{\begin{pmatrix}
\lambda\\x_1\\ \vdots \\ x_r
\end{pmatrix}}{\begin{pmatrix}
\mu\\y_1\\\vdots\\ y_r
\end{pmatrix}} = \frac{1}{1+r\delta^2}\left(\overline{\lambda}\mu+\delta^2\sum\limits_{k=1}^r\is{x_k}{y_k}\right).
    \]
For $k=1,\ldots, r$, let $\iota_k$ be the isometric embedding of $L^2(\mathcal{G})$ into $L^2(\mu_{r-1}(\mathcal{G}))$ 
\[
\iota_k: L^2(\mathcal{G})\ni x \longmapsto \sqrt{\frac{1+r\delta^2}{\delta^2}} \begin{pmatrix}
    0 \\0 \\\vdots \\ 0\\ x\\ 0 \\ \vdots \\ 0
\end{pmatrix} \in L^2(\mu_{r-1}(\mathcal{G})),
\]
where the element $x$ is embedded into the $k$th summand of $\bigoplus\limits_{k=1}^r L^2(\mathcal{G})$. Note that 
\[\iota_k^\ast \begin{pmatrix}
    x_0  \\\vdots  \\ x_k\\ \vdots \\x_r
\end{pmatrix} =\sqrt{\frac{\delta^2}{1+r\delta^2}} x_k.\]
We  also  define the isometric embedding $\iota_0$ of $\CC$ into $L^2(\mu_{r-1}(\mathcal{G}))$,
\[
\iota_0: \CC\ni \lambda \mapsto \sqrt{1+r\delta^2} \begin{pmatrix}
    \lambda \\ 0 \\ \vdots \\ 0
\end{pmatrix}\in L^2(\mu_{r-1}(\mathcal{G})) 
\] and we have \[\iota_0^\ast \begin{pmatrix}
    x_0  \\ \vdots \\x_r
\end{pmatrix} =\sqrt{\frac{1}{1+r\delta^2}} x_0\]
For every $k,l=1,\ldots, r$ we have $\iota_l^\ast \iota_k =\delta_{kl}\id_{L^2(\mathcal{G})}$, and $\iota_0^\ast \iota_0=\id_\CC$. Furthermore, 
\[
\sum\limits_{j=0}^r \iota_j \iota_j^\ast =\id_{L^2(\mu_{r-1}(\mathcal{G}))}  
\]
Denoting $m_0=m_\bullet$ and $m_j=\delta^{-1}m_{\mathcal{G}}$ for $j=1,\ldots r$ we have 
\begin{equation}\label{mmu}
m_{\mu_{r-1}(\mathcal{G})}=\sqrt{1+r\delta^2} \sum\limits_{k=0}^r \iota_k m_k(\iota_k^\ast \otimes \iota_k^\ast),
\end{equation}
and
\[
m_{\mu_{r-1}(\mathcal{G})}^\ast=\sqrt{1+r\delta^2} \sum\limits_{k=0}^r (\iota_k\otimes\iota_k) m_k^\ast \iota_k^\ast.
\]
The quantum adjacency matrix can be written as
\begin{equation}\label{eq:amu}
A_{\mu_{r-1}(\mathcal{G})}=\delta \iota_r \eta_{\mathcal{G}} \iota_0^\ast+\delta\iota_0\eta^\ast_{\mathcal{G}} \iota_r^\ast+ \iota_1 A_{\mathcal{G}}\iota_1^\ast +\sum\limits_{k=1}^{r-1}\left(\iota_k A_\mathcal{G}\iota_{k+1}^\ast + \iota_{k+1} A_\mathcal{G}\iota_k^\ast\right).
\end{equation} 
From this formula we deduce that $A_{\mu_{r-1}(\mathcal{G})}$ is a selfadjoint operator on $L^2(\mu_{r-1}(\mathcal{G}))$. Let us prove that $\psi_{\mu_{r-1}(\mathcal{G})}$ is a $\sqrt{1+r\delta^2}$-form.

\begin{proposition}
\label{prop:delta1}
    The equation 
    \[m_{\mu_{r-1}(\mathcal{G})}m^\ast_{\mu_{r-1}(\mathcal{G})}=(1+r\delta^2)\id_{L^2(\mathcal{G})} \] 
    holds, and in particular $\psi_{\mu_{r-1}(\mathcal{G})}$ is a $\sqrt{1+r\delta^2}$-form.
\end{proposition}
\begin{proof}
We compute 
\begin{align*}
m_{\mu_{r-1}(\mathcal{G})}m^\ast_{\mu_{r-1}(\mathcal{G})}&=(1+r\delta^2) \iota_0  m^{}_{\bullet}  m^\ast_{\bullet}  \iota_0^\ast +\frac{1+r\delta^2}{\delta^2}\sum\limits_{k=1}^r\iota_k  m_{\mathcal{G}}  m_{\mathcal{G}}^\ast \iota_k^\ast\\&=(1+r\delta^2) \iota_0\iota_0^\ast +\frac{1+r\delta^2}{\delta^2}\delta^2\sum\limits_{k=1}^r\iota_k  \iota_k^\ast  = (1+r\delta^2)\id_{L^2(\mu_{r-1}(\mathcal{G}))}. 
\end{align*}
\end{proof}

\begin{proposition}
\label{prop:Myc1}
    For any quantum graph $\mathcal{G}$ and any $r\geq 1$, the generalized Mycielski transformation $\mu_{r-1}(\mathcal{G})$  is a quantum graph. Moreover, if $\mathcal{G}$ was irreflexive (resp. reflexive), then $\mu_{r-1}(\mathcal{G})$ is so. 
\end{proposition}
\begin{proof}
To prove the claim, we have to verify that $\mu_{r-1}(\mathcal{G})$ fulfills all the axioms from Definition~\ref{def:QG}.

First, by Proposition~\ref{prop:delta1} we know that $\psi_{\mu_{r-1}(\mathcal{G})}$ is a $\sqrt{1+r\delta^2}$-form. Moreover, $A_{\mu_{r-1}(\mathcal{G})}\in B(L^2(\mu_{r-1}(\mathcal{G})))$ is a selfadjoint operator (c.f. Eq. \eqref{eq:amu}). 

Next, we compute
\begin{align*}
    m_{\mu_{r-1}(\mathcal{G})}\left(A_{\mu_{r-1}(\mathcal{G})}\otimes A_{\mu_{r-1}(\mathcal{G})}\right)m_{\mu_{r-1}(\mathcal{G})}^\ast = &(r \delta^2 +1)\left(\iota_0 m_\bullet (\iota_0^\ast \otimes\iota_0^\ast)+\frac{1}{\delta} \sum\limits_{k=1}^{r}\iota_k m_{\mathcal{G}}(\iota_k^\ast\otimes \iota_k^\ast )\right) \\
    &\times \left(A_{\mu_{r-1}(\mathcal{G})}\otimes A_{\mu_{r-1}(\mathcal{G})}\right)\left((\iota_0\otimes\iota_0)m_\bullet^\ast \iota_0^\ast + \frac{1}{\delta}\sum\limits_{k=1}^{r}(\iota_k\otimes \iota_k)m_{\mathcal{G}}^\ast \iota_k^\ast\right).
\end{align*}
Since $m_{\mu_{r-1}(\mathcal{G})}$ and $m_{\mu_{r-1}(\mathcal{G})}^\ast$ contain only tensors of the form $\iota_j^\ast \otimes \iota_j^\ast$ and $\iota_j \otimes \iota_j$ ($j=0,\ldots,r$), respectively (i.e. diagonal in the subscripts), the only non-zero contribution from $A_{\mu_{r-1}(\mathcal{G})}\otimes A_{\mu_{r-1}(\mathcal{G})}$ is
\begin{align*}
    &\delta^2(\iota_r\otimes \iota_r)(\eta_{\mathcal{G}}\otimes \eta_{\mathcal{G}})(\iota_0^\ast \otimes \iota_0^\ast)+\delta^2(\iota_0\otimes \iota_0)(\eta^\ast_{\mathcal{G}}\otimes \eta^\ast_{\mathcal{G}})(\iota_r^\ast \otimes \iota_r^\ast)+(\iota_1\otimes \iota_1)(A_{\mathcal{G}}\otimes A_{\mathcal{G}})(\iota_1^\ast \otimes \iota_1^\ast)\\
    +& \sum\limits_{k=1}^{r-1}\left[(\iota_k\otimes \iota_k)(A_{\mathcal{G}}\otimes A_{\mathcal{G}})(\iota_{k+1}^\ast \otimes \iota_{k+1}^\ast)+(\iota_{k+1}\otimes\iota_{k+1})(A_{\mathcal{G}}\otimes A_{\mathcal{G}})(\iota_k^\ast \otimes\iota_k^\ast)\right].
\end{align*}
Therefore,
\begin{align*}
     m_{\mu_{r-1}(\mathcal{G})}\left(A_{\mu_{r-1}(\mathcal{G})}\otimes A_{\mu_{r-1}(\mathcal{G})}\right)m_{\mu_{r-1}(\mathcal{G})}^\ast = &(1+r\delta^2)\left\{\delta \iota_0 m_\bullet (\eta^\ast_{\mathcal{G}}\otimes \eta^\ast_{\mathcal{G}})m_{\mathcal{G}}^\ast \iota_r^\ast+\delta \iota_r m_{\mathcal{G}}(\eta_{\mathcal{G}}\otimes\eta_{\mathcal{G}})m_\bullet^\ast \iota_0^\ast\right.\\
     &+\left.\delta^{-2}\iota_1 m_{\mathcal{G}} (A_{\mathcal{G}}\otimes A_{\mathcal{G}})m_{\mathcal{G}}^\ast\iota_1^\ast \right.\\
     & + \left.
     \delta^{-2}\sum\limits_{k,p=1}^r\sum\limits_{l=1}^{r-1}\iota_k m_{\mathcal{G}}(\iota_k^\ast\otimes \iota_k^\ast)\left[(\iota_l\otimes\iota_l)(A_{\mathcal{G}}\otimes A_{\mathcal{G}})(\iota_{l+1}^\ast\otimes \iota_{l+1}^\ast)\right.\right.\\
     &\hspace{1.8cm}\left.\left.+(\iota_{l+1}\otimes \iota_{l+1})(A_{\mathcal{G}}\otimes A_{\mathcal{G}})(\iota_l^\ast\otimes\iota_l^\ast)\right](\iota_p\otimes \iota_p)m_{\mathcal{G}}^\ast \iota_p
     \right\}.
\end{align*}
This can be further simplified into
\begin{align*}
     m_{\mu_{r-1}(\mathcal{G})}\left(A_{\mu_{r-1}(\mathcal{G})}\otimes A_{\mu_{r-1}(\mathcal{G})}\right)m_{\mu_{r-1}(\mathcal{G})}^\ast = &(1+r\delta^2) \Big\{\delta \iota_0 m_\bullet (\eta^\ast_{\mathcal{G}}\otimes \eta^\ast_{\mathcal{G}})m_{\mathcal{G}}^\ast \iota_r^\ast+\delta \iota_r m_{\mathcal{G}}(\eta_{\mathcal{G}}\otimes\eta_{\mathcal{G}})m_\bullet^\ast \iota_0^\ast \\
     &+ \delta^{-2}\iota_1 m_{\mathcal{G}} (A_{\mathcal{G}}\otimes A_{\mathcal{G}})m_{\mathcal{G}}^\ast\iota_1^\ast \\
     & +  \delta^{-2}\sum\limits_{k=1}^{r-1}\left(\iota_k m_{\mathcal{G}}(A_{\mathcal{G}}\otimes A_{\mathcal{G}})m_{\mathcal{G}}^\ast\iota_{k+1}^\ast + \iota_{k+1}m_{\mathcal{G}}(A_{\mathcal{G}}\otimes A_{\mathcal{G}})m_{\mathcal{G}}^\ast\iota_{k}^\ast\right)
     \Big\}.
\end{align*}
Since $m_{\mathcal{G}} (A_{\mathcal{G}}\otimes A_{\mathcal{G}})m_{\mathcal{G}}^\ast=\delta^2 A_{\mathcal{G}}$ and $m_{\mathcal{G}}(\eta_{\mathcal{G}}\otimes\eta_{\mathcal{G}})m_\bullet^\ast = \eta_{\mathcal{G}}$, we get
\begin{align*}
     m_{\mu_{r-1}(\mathcal{G})}\left(A_{\mu_{r-1}(\mathcal{G})}\otimes A_{\mu_{r-1}(\mathcal{G})}\right)m_{\mu_{r-1}(\mathcal{G})}^\ast =(r \delta^2+1)A_{\mu_{r-1}(\mathcal{G})}.
\end{align*}

In order to show that
\begin{align*}
       (\id\otimes \eta_{\mu_{r-1}(\mathcal{G})}^\ast m_{\mu_{r-1}(\mathcal{G})})(\id\otimes A_{\mu_{r-1}(\mathcal{G})}\otimes \id)(m_{\mu_{r-1}(\mathcal{G})}^\ast \eta_{\mu_{r-1}(\mathcal{G})} \otimes\id)=A_{\mu_{r-1}(\mathcal{G})} 
\end{align*}
notice that since $\eta_{\mu_{r-1}(\mathcal{G})}=\frac{1}{\sqrt{1+r\delta^2}}\left(\iota_0\eta_\bullet + \delta \sum\limits_{k=1}^r \iota_k \eta_{\mathcal{G}}\right)$, and Eq. \eqref{mmu} holds, it follows that
\begin{align*}
    \eta^\ast_{\mu_{r-1}(\mathcal{G})}m_{\mu_{r-1}(\mathcal{G})}=\eta_\bullet^\ast m_\bullet(\iota_0^\ast \otimes \iota_0^\ast) +\sum\limits_{k=1}^r \eta^\ast_{\mathcal{G}}m_{\mathcal{G}}(\iota_k^\ast \otimes\iota_k^\ast),
\end{align*}
and
\begin{align*}
m^\ast_{\mu_{r-1}(\mathcal{G})}\eta_{\mu_{r-1}(\mathcal{G})}=(\iota_0 \otimes\iota_0)m_\bullet^\ast \eta_\bullet +\sum\limits_{k=1}^r(\iota_k\otimes\iota_k)m^\ast_{\mathcal{G}}\eta_{\mathcal{G}}.
\end{align*}
Therefore,
\begin{align*}
    (\id\otimes \eta_{\mu_{r-1}(\mathcal{G})}^\ast m_{\mu_{r-1}(\mathcal{G})})&(\id\otimes A_{\mu_{r-1}(\mathcal{G})}\otimes \id)(m_{\mu_{r-1}(\mathcal{G})}^\ast \eta_{\mu_{r-1}(\mathcal{G})} \otimes\id)\\
    =& \left[ \id \otimes \eta_\bullet^\ast m_\bullet (\iota_0^\ast \otimes\iota_0^\ast) + \id \otimes \sum\limits_{k=1}^r \eta^\ast_{\mathcal{G}}m_{\mathcal{G}}(\iota_k^\ast \otimes \iota_k^\ast) \right]\\
    &\times \Biggl[\delta \,\id\otimes \iota_r \eta_{\mathcal{G}} \iota_0^\ast\otimes \id+\delta\,\id\otimes\iota_0\eta^\ast_{\mathcal{G}} \iota_r^\ast\otimes\id+ \id\otimes\iota_1 A_{\mathcal{G}}\iota_1^\ast\otimes\id \Biggr.\\ &\left.+\sum\limits_{k=1}^{r-1}\left(\id\otimes\iota_k A_\mathcal{G}\iota_{k+1}^\ast \otimes\id+ \id\otimes\iota_{k+1} A_\mathcal{G}\iota_k^\ast\otimes\id\right)\right]\\
    &\times \left[(\iota_0\otimes\iota_0)m_\bullet^\ast \eta_\bullet \otimes\id  +\sum\limits_{k=1}^r (\iota_k\otimes\iota_k)m^\ast_{\mathcal{G}}\eta_{\mathcal{G}}\otimes\id\right].
\end{align*}
The only non-zero contributions to the above expression are:
\begin{align*}
    &(\id\otimes \eta_{\mu_{r-1}(\mathcal{G})}^\ast m_{\mu_{r-1}(\mathcal{G})})(\id\otimes A_{\mu_{r-1}(\mathcal{G})}\otimes \id)(m_{\mu_{r-1}(\mathcal{G})}^\ast \eta_{\mu_{r-1}(\mathcal{G})} \otimes\id)\\
    &=\delta \left[\id\otimes \eta_\bullet^\ast m_\bullet (\iota_0^\ast \otimes \iota_0^\ast)\right]\left[\id \otimes \iota_0 \eta^\ast_{\mathcal{G}}\iota^\ast_r\otimes \id\right]\left[(\iota_r\otimes \iota_r)m^\ast_{\mathcal{G}}\eta_{\mathcal{G}}\otimes\id\right] \\
    & +\delta\left[\id \otimes \eta^\ast_{\mathcal{G}}m_{\mathcal{G}}(\iota^\ast_r\otimes\iota^\ast_r)\right]\left[\id \otimes\iota_r \eta_{\mathcal{G}}\iota^\ast_0\otimes\id\right]\left[(\iota_0\otimes\iota_0)m^\ast_{\bullet}\eta_\bullet \otimes\id\right]\\
    &+\left[\id\otimes \eta^\ast_{\mathcal{G}}m_{\mathcal{G}}(\iota^\ast_1\otimes\iota^\ast_1)\right]\left[\id\otimes \iota_1 A_{\mathcal{G}}\iota^\ast_1 \otimes\id\right]\left[(\iota_1\otimes\iota_1)m^\ast_{\mathcal{G}}\eta_{\mathcal{G}}\otimes\id\right]\\
    & +\sum\limits_{l,p=1}^r\sum\limits_{k=1}^{r-1}\left[\id\otimes \eta^\ast_{\mathcal{G}}m_{\mathcal{G}}(\iota^\ast_l\otimes\iota^\ast_l)\right]\left[\id\otimes \iota_k A_{\mathcal{G}}\iota^\ast_{k+1}\otimes \id+\id \otimes\iota_{k+1}A_{\mathcal{G}}\iota^\ast_k\otimes\id\right]\left[(\iota_p\otimes\iota_p)m^\ast_{\mathcal{G}}\eta_{\mathcal{G}}\otimes\id\right].
\end{align*}
This can be further reduced into
\begin{align*}
    (\id&\otimes \eta_{\mu_{r-1}(\mathcal{G})}^\ast m_{\mu_{r-1}(\mathcal{G})})(\id\otimes A_{\mu_{r-1}(\mathcal{G})}\otimes \id)(m_{\mu_{r-1}(\mathcal{G})}^\ast \eta_{\mu_{r-1}(\mathcal{G})} \otimes\id)\\
    &=\delta(\id\otimes \eta_\bullet^\ast m_\bullet)(\iota_r\otimes \id \otimes\iota^\ast_0)(\id \otimes \eta^\ast_{\mathcal{G}}\otimes\id)(m^\ast_{\mathcal{G}}\otimes \id)(\eta_{\mathcal{G}}\otimes \id) \\
    &+ \delta(\id\otimes\eta^\ast_{\mathcal{G}})(\id\otimes m_{\mathcal{G}})(\id \otimes \eta_{\mathcal{G}}\otimes \id)(\iota_0\otimes\id \otimes \iota^\ast_r)(m_\bullet^\ast \eta_\bullet\otimes\id)\\
    &+(\id\otimes\eta_{\mathcal{G}}^\ast m_{\mathcal{G}})(\iota_1 \otimes A_{\mathcal{G}}\otimes\iota^\ast_{1})(m_{\mathcal{G}}^\ast \eta_{\mathcal{G}}\otimes\id)
    \\
    & +\sum\limits_{k=1}^{r-1}\left[(\id\otimes\eta_{\mathcal{G}}^\ast m_{\mathcal{G}})(\iota_k \otimes A_{\mathcal{G}}\otimes\iota^\ast_{k+1})(m_{\mathcal{G}}^\ast\eta_{\mathcal{G}}\otimes\id)+(\id\otimes\eta_{\mathcal{G}}^\ast m_{\mathcal{G}})(\iota_{k+1} \otimes A_{\mathcal{G}}\otimes\iota^\ast_{k})(m_{\mathcal{G}}^\ast \eta_{\mathcal{G}}\otimes\id)\right].
\end{align*}
Notice that,
\begin{align*}
    (\id\otimes \eta_\bullet^\ast m_\bullet)(\iota_r\otimes \id \otimes\iota^\ast_0)(\id \otimes \eta^\ast_{\mathcal{G}}\otimes\id)&(m^\ast_{\mathcal{G}}\otimes \id)(\eta_{\mathcal{G}}\otimes \id)=\\&=(\iota_r\otimes \id \otimes\iota^\ast_0)(\id \otimes \eta^\ast_{\mathcal{G}}\otimes\id)(m^\ast_{\mathcal{G}}\otimes \id)(\eta_{\mathcal{G}}\otimes \id)\\&=\iota_r ((\id \otimes\eta^\ast_{\mathcal{G}})m^\ast_{\mathcal{G}}\eta_{\mathcal{G}})\iota_0^\ast=\iota_r \eta_{\mathcal{G}}\iota_0^\ast,
\end{align*}
where to justify the first step we observe that the map $\eta_\bullet^\ast m_\bullet:\CC\otimes\CC\to \CC$ coincides with  the canonical identification of $\CC\otimes\CC$ and  $\CC$ and in the second step we used Lemma~\ref{lem:eta}. Similarly,
\begin{align*}
    (\id\otimes\eta^\ast_{\mathcal{G}})(\id\otimes m_{\mathcal{G}})(\id \otimes \eta_{\mathcal{G}}\otimes \id)(\iota_0\otimes\id \otimes \iota^\ast_r)(m_\bullet^\ast \eta_\bullet\otimes\id)=\iota_0 \eta^\ast_{\mathcal{G}}\iota_r^\ast.
\end{align*}
Moreover,
\begin{align*}
    (\id\otimes \eta^\ast_{\mathcal{G}}m_{\mathcal{G}})(\iota_k\otimes A_{\mathcal{G}}\otimes \iota^\ast_l)(m_{\mathcal{G}}^\ast \eta_{\mathcal{G}}\otimes \id)&=\iota_k(\id\otimes\eta^\ast_{\mathcal{G}}m_{\mathcal{G}})(\id\otimes A_{\mathcal{G}}\otimes \id)(m^\ast_{\mathcal{G}}\eta_{\mathcal{G}}\otimes \id)\iota_l^\ast=\iota_k A_{\mathcal{G}}\iota_l^\ast,
\end{align*}
for every $k$ and $l$ which follows from Eq.~\eqref
{eq2:defA} that is satisfied by   $A_{\mathcal{G}}$. All this demonstrates that
\begin{align*}
      (\id\otimes \eta_{\mu_{r-1}(\mathcal{G})}^\ast m_{\mu_{r-1}(\mathcal{G})})&(\id\otimes A_{\mu_{r-1}(\mathcal{G})}\otimes \id)(m_{\mu_{r-1}(\mathcal{G})}^\ast \eta_{\mu_{r-1}(\mathcal{G})} \otimes\id)=A_{\mu_{r-1}(\mathcal{G})}.
\end{align*}
In order to  check  the (ir)reflexivity of $\mu_{r-1}(\mathcal{G})$ we compute
\begin{align*}
    m_{\mu_{r-1}(\mathcal{G})}\left(A_{\mu_{r-1}(\mathcal{G})}\otimes \id\right)m^\ast_{\mu_{r-1}(\mathcal{G})}&=(1+r\delta^2)\left[\iota_0m_\bullet (\iota_0^\ast \otimes\iota_0^\ast)+\frac{1}{\delta}\sum\limits_{k=1}^r\iota_k m_{\mathcal{G}}(\iota_k^\ast \otimes \iota_k^\ast)\right]\\
    &\times \Biggl[\delta\iota_r \eta_{\mathcal{G}}\iota_0^\ast \otimes \id + \delta \iota_0\eta^\ast_{\mathcal{G}}\iota_r^\ast \otimes\id +\iota_1 A_{\mathcal{G}}\iota_1^\ast\otimes \id \Biggr.\\ &\hspace{1cm}+\left.\sum\limits_{k=1}^r\left(\iota_k A_{\mathcal{G}}\iota_{k+1}^\ast \otimes \id + \iota_{k+1}A_{\mathcal{G}}\iota_k^\ast \otimes \id\right)\right]\\
    &\times \left[(\iota_0\otimes \iota_0)m_\bullet^\ast \iota_0^\ast+\frac{1}{\delta}\sum\limits_{k=1}^r(\iota_k\otimes \iota_k)m^\ast_{\mathcal{G}}\iota^\ast_k\right]\\
    &=\frac{1+r\delta^2}{\delta^2}\iota_1 m_{\mathcal{G}}(A_{\mathcal{G}}\otimes \id)m^\ast_{\mathcal{G}}\iota^\ast_1=\begin{cases}0, \quad \mathcal{G} - \text{irreflexive},\\
    (1+r\delta^2)\id, \quad \mathcal{G} - \text{reflexive}.\end{cases}
\end{align*}
\end{proof}

\section{Chromatic numbers under Mycielski transformation}
\label{sec:chromatic}
\begin{definition}\label{def:color} \cite[Definition~2.16]{Ganesan21} Let $\mathcal{G}$ be an irreflexive quantum graph and $S_{\mathcal{G}}$ the associated operator {space} as described in Eq.~\eqref{eq:opsys}. We say that $\mathcal{G}$ possesses a quantum $c$-coloring if there exists a finite von Neumann algebra $\mathcal{N}$  and a partition of unity  $\{P_a\}_{a=1}^c\in \C(\mathcal{G})\otimes \mathcal{N}$ such that 
\[\forall_{  1\leq a\leq c} \, \forall _{X\in S_{\mathcal{G}}}:\, P_a(X\otimes \I_\mathcal{N})P_a=0.\]
The quantum chromatic number  is  defined as (c.f.~\cite{BGH}) 
\begin{equation}
    \chi_q(\mathcal{G})=\min\{c\in\mathbb{N}| \, \exists \, c\mbox{-colouring with } \dim(\mathcal{N})<\infty \},
\end{equation}
while the classical chromatic number for a quantum graph is
\begin{equation}
\chi_\mathrm{loc}(\mathcal{G})=\min\{c\in\mathbb{N}| \, \exists \, c\mbox{-colouring with } \dim(\mathcal{N})=1 \}.
\end{equation}
\end{definition}

For a classical graph $G$, it is known that $\chi_{\text{loc}}(\mu(G))=\chi_{\text{loc}}(G)+1$. We now study this type of relation in the quantum framework.

\begin{proposition}\label{prop:ineq}
     For every irreflexive quantum graph $\mathcal{G}$, $r\geq 1$ and $t\in\{\text{loc},q\}$, $\chi_{t}(\mu_{r-1}(\mathcal{G}))\leq \chi_{t}(\mathcal{G})+1$. 
\end{proposition}

\begin{proof}
The proof for $t=q$ is an amplification of the proof for $t=\text{loc}$ which we give below.   
We show that if $\{P_k\}_{k=1}^c\subset \C(\mathcal{G})$  is a $c$-colouring of $\mathcal{G}$, then there exists a $c+1$-colouring $\{Q_k\}_{k=0}^c\subset \C(\mu_{r-1}(\mathcal{G}))$ of $\mu_{r-1}(\mathcal{G})$. Viewing $\C(\mu_{r-1}(\mathcal{G}))$ as an algebra acting on $L^2(\mu_{r-1}(\mathcal{G}))$  we define self-adjoint projections $Q_0=\iota_0\iota^\ast _0$ and  
$Q_i=\sum\limits_{j=1}^r\iota_j P_i\iota_j^\ast $ for $i=1,2,\ldots, c$ which incidentally form a partition of unity of $\C(\mu_{r-1}(\mathcal{G}))$:
\begin{align*}
\sum\limits^r_{i=0}Q_0 = \iota_0\iota
_0^\ast  +\sum\limits^r_{i,j=1}\iota_j P_i\iota_j^\ast = \iota_0\iota_0^\ast  +\sum\limits^r_{ j=1}\iota_j  \iota_j^\ast = \I.
\end{align*} In order to conclude that $\{Q_k\}_{k=0}^c$ is $c+1$-colouring of $\mu_{r-1}(\mathcal{G})$, take $X\in S_{\mu_{r-1}(\mathcal{G})}$, i.e.
\[X = m_{\mu_{r-1}(\mathcal{G})}(A_{\mu_{r-1}(\mathcal{G})}\otimes Y)m_{\mu_{r-1}(\mathcal{G})}^\ast \] for some $Y\in B(L^2(\mu_{r-1}(\mathcal{G})))$. Remembering that $L^2(\mu_{r-1}(\mathcal{G}))$ can be identified with the direct sum of $\CC$ and $r$ copies of  $L^2( \mathcal{G})$ (c.f. Definition~\ref{rmyciel}) we shall use matrix notation $(Y_{ij})_{i,j=0,\ldots,r}$, where $Y_{00}\in\CC$, $Y_{j0},Y_{0j}^\ast\in L^2(\mathcal{G})$ for $j=1,\ldots,r$, and $Y_{ij}\in B(L^2(\mathcal{G}))$ for $i,j=1,\ldots,r$. 

Let us first show that $Q_0XQ_0 = 0$:
\begingroup
\allowdisplaybreaks
\begin{align*}
    Q_0 m_{\mu_{r-1}(\mathcal{G})}(A_{\mu_{r-1}(\mathcal{G})}\otimes Y)m^\ast_{\mu_{r-1}(\mathcal{G})}Q_0&=  \iota_0\iota^\ast _0 m_{\mu_{r-1}(\mathcal{G})}(A_{\mu_{r-1}(\mathcal{G})}\otimes Y)\sqrt{1+r\delta^2} \sum\limits_{k=0}^r (\iota_k\otimes\iota_k) m_k^\ast \iota_k^\ast\iota_0\iota^\ast _0\\&= \sqrt{1+r\delta^2} \iota_0\iota^\ast _0 m_{\mu_{r-1}(\mathcal{G})}(A_{\mu_{r-1}(\mathcal{G})}\otimes Y)  (\iota_0\otimes\iota_0) m_0^\ast \iota^\ast _0\\&=\sqrt{1+r\delta^2} \iota_0\iota^\ast _0 m_{\mu_{r-1}(\mathcal{G})} (\delta \iota_r \eta_{\mathcal{G}}\otimes Y\iota_0)m^\ast_0\eta_0\\
    &=\delta (1+r\delta^2)\iota_0\iota^\ast _0\iota_r m_r (\eta_{\mathcal{G}}\otimes Y_{r0})m^\ast_0\eta_0=0,
\end{align*}
since $\iota_0^\ast \iota_r=0$.

Similarly we shall show that $Q_iXQ_i=0$ for all $i=1,\ldots,c$:
\begin{align*}\allowdisplaybreaks
    Q_i& m_{\mu_{r-1}(\mathcal{G})}(A_{\mu_{r-1}(\mathcal{G})}\otimes Y)m_{\mu_{r-1}(\mathcal{G})}^\ast Q_i\\
    &= \sum\limits_{j,k=1}^r\iota_k P_i\iota_k^\ast  m_{\mu_{r-1}(\mathcal{G})}(A_{\mu_{r-1}(\mathcal{G})}\otimes Y)\sqrt{1+r\delta^2} \sum\limits_{p=0}^r (\iota_p\otimes\iota_p) m_p^\ast \iota_p^\ast  \iota_j P_i\iota_j^\ast \\
    &=\sqrt{1+r\delta^2}\sum\limits_{j=1}^r Q_i m_{\mu_{r-1}(\mathcal{G})}(A_{\mu_{r-1}(\mathcal{G})}\otimes Y)(\iota_j\otimes\iota_j)m^\ast_j P_i\iota^\ast_j\\
    &=\sqrt{1+r\delta^2}Q_i m_{\mu_{r-1}(\mathcal{G})}\Biggl\{(\delta\iota_0 \eta^\ast_{\mathcal{G}}\otimes Y \iota_r)m^\ast_r P
_i\iota_r^\ast + \sum\limits_{j=1}^{r}(\iota_1 A_{\mathcal{G}}\iota_1^\ast \otimes Y )(\iota_j\otimes\iota_j)m_j^\ast P_i\iota_j^\ast\Biggr. \\ &\hspace{2.5cm}\Biggl. +\sum\limits_{j=1}^r\sum\limits_{p=1}^{r-1}\left[(\iota_pA_{\mathcal{G}}\iota^\ast_{p+1}+\iota_{p+1}A_{\mathcal{G}}\iota_p^\ast)\otimes Y\right](\iota_j\otimes\iota_j)m_j^\ast P_i\iota_j^\ast\Biggr\}\\
    &=   \sum\limits_{l=0}^r Q_i \iota_l m_l (\iota_l^\ast \otimes\iota_l^\ast)\Biggl\{(\delta\iota_0 \eta^\ast_{\mathcal{G}}\otimes Y \iota_r)m^\ast_r P_i\iota_r^\ast +(\iota_1 A_{\mathcal{G}}\iota_1^\ast \otimes Y\iota_1)m_1^\ast P_i\iota_1^\ast\Biggr.\\
    &\hspace{2.5cm}\Biggl.+ \sum\limits_{j=1}^{r-1}\left[(\iota_jm_j(A_{\mathcal{G}}\otimes Y \iota_{j+1}))m^\ast_{j+1}P_i\iota_{j}^\ast+(\iota_{j+1}A_{\mathcal{G}}\otimes Y\iota_j)m^\ast_j P_i\iota_{j+1}^\ast\right]\Biggr\}\\
    &=Q_i\iota_0 m_0   (\delta\iota_0 \eta^\ast_{\mathcal{G}}\otimes Y_{0,r})m^\ast_r P_i\iota_r^\ast 
    + Q_i\iota_1 m_1 (A_{\mathcal{G}}\otimes Y_{11})m^\ast_1 P_i \iota_1^\ast  \\
    &\hspace{0.2cm}+  \sum\limits_{l=1}^r\sum\limits_{j=1}^{r-1} Q_i \iota_l m_l (\iota_l^\ast \otimes\iota_l^\ast) \left[(\iota_jm_j(A_{\mathcal{G}}\otimes Y \iota_{j+1}))m^\ast_{j+1}P_i\iota_j^\ast+i\iota^\ast_j+(\iota_{j+1}A_{\mathcal{G}}\otimes Y\iota_j)m^\ast_jP_i\iota_{j+1}^\ast\right] \\
    &= Q_i\iota_1 m_1 (A_{\mathcal{G}}\otimes Y_{11})m^\ast_1 P_i\iota^\ast_1   + Q_i \sum\limits_{j=1}^{r-1}\left[\iota_j m_j (A_{\mathcal{G}}\otimes Y_{j\,j+1})m^\ast_{j+1}P_i\iota_j^\ast +\iota_{j+1} m_{j+1}(A_{\mathcal{G}}\otimes Y_{j+1\,j})m^\ast_jP_i\iota_{j+1}^\ast\right] 
    \\
    &=\iota_1 P_i m_1 (A_{\mathcal{G}}\otimes Y_{11})m^\ast_1 P_i\iota^\ast_1   +  \sum\limits_{j=1}^{r-1}\left[\iota_jP_i m_j (A_{\mathcal{G}}\otimes Y_{j\,j+1})m^\ast_{j+1}P_i\iota_j^\ast +\iota_{j+1}P_i m_{j+1}(A_{\mathcal{G}}\otimes Y_{j+1\,j})m^\ast_jP_i\iota_{j+1}^\ast\right] 
    \\
   &=\iota_1 P_i m_{\mathcal{G}} (A_{\mathcal{G}}\otimes Y_{11})m^\ast_{\mathcal{G}} P_i\iota^\ast_1   +  \sum\limits_{j=1}^{r-1}\left[\iota_jP_i m_{\mathcal{G}}(A_{\mathcal{G}}\otimes Y_{j\,j+1})m^\ast_{\mathcal{G}}P_i\iota_j^\ast +\iota_{j+1}P_i m_{\mathcal{G}}(A_{\mathcal{G}}\otimes Y_{j+1\,j})m^\ast_{\mathcal{G}}P_i\iota_{j+1}^\ast\right]=0
\end{align*}
\endgroup
where in the sixth equality we used $Q_i\iota_0 = 0$ for all $i\geq 1$ and in the last equality we  refer to the fact that  $\{P_k\}_{k=1}^c\subset \C(\mathcal{G})$  is a $c$-colouring of $\mathcal{G}$.
\end{proof}

It is known that in contrast to the case $r=2$,   the generalized Mycielski construction $\mu_{r-1}$ for $r>2$ may leave classical chromatic number unchanged. Let us restrict our attention to the $r=2$-case for quantum graph $\mathcal{G}$ and suppose  that  $\{P_{k1}\}_{k=0}^{c'}$ is a $c'+1$-colouring of $\mu(\mathcal{G})$.  Since $\C(\mu(\mathcal{G})) = \CC\oplus\C( \mathcal{G} )\oplus\C(\mathcal{G})$ we can, without loss of generality  assume that  $P_0 = (1,P_{01},P_{02})$ and $P_{k1} = (0,P_{k1},P_{k2})$ for $k=1,2,\ldots,c'$. Our goal is to prove that under a mild commutativity condition $\chi_{\text{loc}}(\mu(\mathcal{G})) =\chi_{\text{loc}}(\mathcal{G})+1$. First, we observe: 
\begin{lemma}
\label{lemma:P2}
With the notation above, $P_{02} = 0$.
\end{lemma}
\begin{proof}
Consider $Y\in B(L^2(\mu(\mathcal{G})))$. Remembering that $Y = (Y_{ij})_{i,j=0,1,2}$ where  $Y_{00}\in\CC$, $Y_{10},Y_{20}\in L^2(\mathcal{G})$, we compute 
\begin{align*}
P_0 &m_{\mu(\mathcal{G})}(A_{\mu(\mathcal{G})} \otimes Y)m_{\mu(\mathcal{G})}^\ast P_0\begin{pmatrix}1\\0\\0\end{pmatrix}  = (1+2\delta^2)P_0 m_{\mu(\mathcal{G})}(A_{\mu(\mathcal{G})}\otimes Y)\left(\begin{pmatrix}1\\0\\0\end{pmatrix}\otimes \begin{pmatrix}1\\0\\0\end{pmatrix}\right)\\&= (1+2\delta^2)P_0 m_{\mu(\mathcal{G})}\left(\begin{pmatrix}0\\0\\\I_{\mathbb{G}}\end{pmatrix}\otimes Y\begin{pmatrix}1\\0\\0\end{pmatrix}\right) \\&= (1+2\delta^2)P_0 m_{\mu(\mathcal{G})}\left(\begin{pmatrix}0\\0\\\I_{\mathbb{G}}\end{pmatrix}\otimes  \begin{pmatrix}Y_{00}\\Y_{10}\\Y_{20}\end{pmatrix}\right) \\&= (1+2\delta^2)P_0\begin{pmatrix}0\\0\\Y_{20}\end{pmatrix} = (1+2\delta^2)\begin{pmatrix}0\\0\\P_{02}Y_{20}\end{pmatrix} \equiv 0
\end{align*} and since this holds for all $Y_{20}\in L^2(\mathcal{G})$, we conclude that $P_{02} = 0$. 
\end{proof}
\begin{proposition}
    Let $\mathcal{G}$ be a quantum graph with  $c'+1=\chi_{\text{loc}}(\mu(\mathcal{G}))$ and suppose that    there exists $c'+1$-colouring   $\{P_{k1}\}_{k=0}^{c'}$ of $\mu(\mathcal{G})$ with the property that $P_{01}P_{l2} =P_{l2}P_{01}$, for all $l\in\{1,\ldots, c'\}$. Then $\chi_{\text{loc}}( \mathcal{G} ) = \chi_{\text{loc}}(\mu(\mathcal{G}))-1$.
\end{proposition}
\begin{proof}
    By Proposition~\ref{prop:ineq} it is enough to construct a $c'$ coloring $\{Q_l\}_{l=1}^{c'}$ of $\mathcal{G}$. By Lemma~\ref{lemma:P2} we can assume that $P_0=(1,P_{01},0)$ and $P_{k1}=(0,P_{k1},P_{k2})$ for $k=1,\ldots, c'$. Note that $P_{12}+\ldots+P_{c'2}=\I$. We define $Q_l = P_{l1}+P_{01}P_{l2}$ and note that 
\[Q_1+\ldots +Q_{c'} =\I .\] Moreover, under the assumed condition, we have \[Q_l^2 = Q_l = Q_l^\ast\]

Next, we compute assuming that $l>0$,
\begin{align*}
0=P_lA_{\mu(\mathcal{G})}P_l\begin{pmatrix}\lambda\\x\\y\end{pmatrix}&=P_lA_{\mu(\mathcal{G})}\begin{pmatrix}0\\P_{l1}x\\P_{l2}y\end{pmatrix}\\&=\begin{pmatrix}0\\P_{l1}A_{ \mathcal{G} }(P_{l1}x+P_{l2}y)\\P_{l2}A_{ \mathcal{G} }P_{l1}x\end{pmatrix}
\end{align*} for all $x,y\in L^2(\mathcal{G})$. Hence we have $P_{l2}A_{ \mathcal{G} }P_{l1}=0 = P_{l1}A_{ \mathcal{G} } P_{l1}$ for all $l>0$. Next,
\begin{align*}
P_0A_{\mu(\mathcal{G})}P_0\begin{pmatrix}1\\x\\y\end{pmatrix}&=P_0A_{\mu(\mathcal{G})}\begin{pmatrix}1\\P_{01}x\\0\end{pmatrix}\\&=\begin{pmatrix}0\\P_{01}A_{ \mathcal{G} }P_{01}x\\0\end{pmatrix}
\end{align*} for all $x\in L^2(\mathcal{G})$. Hence   $P_{l2}A_{ \mathcal{G} }P_{l1}=0 = P_{l1}A_{ \mathcal{G} } P_{l1}$ for all $l\geq 0$ which used in the next computation (under the assumption $P_{01}P_{l2}=P_{l2}P_{01}$) yields
\begin{align*}
Q_lA_{\mathcal{G}}Q_l &=(P_{l1}+P_{01}P_{l2})A_{\mathcal{G}}(P_{l1}+P_{l2}P_{01})\\& = P_{l1}A_{\mathcal{G}}P_{l1}+P_{l1}A_{\mathcal{G}}P_{l2}P_{01}+P_{01}P_{l2}A_{\mathcal{G}}P_{l1}+P_{l2}P_{01}A_{\mathcal{G}}P_{01}P_{l2}\\&=0\end{align*}
Thus if $P_{01}P_{l2} =P_{l2}P_{01}$ for all $l\in \{0,1,\ldots,c'\}$, then $\{Q_l\}_{l=1}^{c'}$ is a $c'$-colouring.
\end{proof}
 \begin{question}
Is there a quantum graph $\mathcal{G}$ for which $\chi_{\text{loc}}(\mu(\mathcal{G}))=\chi_{\text{loc}}(\mathcal{G})$?
\end{question} 
{
The  concept of quantum graphs homomorphism will be useful in this and the next sections.
\begin{definition} \label{def:morphism}(\cite[Definition~7]{Stahlke}) We say that there exists a homomorphism between quantum graphs $\mathcal{G}$ and $\mathcal{F}$, and write $\mathcal{G}\rightarrow \mathcal{F}$, if there exists a Hilbert space $\mathsf{H}$ and an isometry $\mathcal{J}: L^2(\mathcal{G})\rightarrow L^2(\mathcal{F})\otimes \mathsf{H}$ s.t. $\mathcal{J}S_{\mathcal{G}}\mathcal{J}^* \subseteq S_{\mathcal{F}}\otimes B(\mathsf{H})$.
\end{definition}
By \cite[Theorem 8]{Stahlke} every homomorphism in the above sense between classical graphs  corresponds to a homomorphism between these graphs.
 
\begin{definition} (\cite[Section~III, p.~6]{Stahlke}) We say that $\mathcal{G}$ is a quantum subgraph of $\mathcal{F}$ if there exists an isometry $\mathcal{J}:L^2(\mathcal{G})\rightarrow L^2(\mathcal{F})$ s.t. $\mathcal{J}S_\mathcal{G} \mathcal{J}^* \subseteq S_{\mathcal{F}}$. 
\end{definition}
\begin{remark}
\label{rk:mu}
    {Let us note that $\mathcal{G}$ is a quantum subgraph of Mycielskian $\mu_{r-1}(\mathcal{G})$, for any $r\geq 1$, where the isometry $J:L^2(\mathcal{G})\to L^2(\mu_{r-1}(\mathcal{G})) $ is given by $J(x) =  \iota_1(x)$.}
\end{remark}
Note that if $\mathcal{G}$ is a quantum subgraph of $\mathcal{F}$, then in particular $\mathcal{G}\rightarrow \mathcal{F}$. As a direct consequence of the monotonicity property of quantum chromatic numbers \cite[Proposition~6.4]{BGH} and Remark~\ref{rk:mu} we have the following
} 
\begin{proposition}
\label{prop:mono_chrom}
For any irreflexive quantum graph $\mathcal{G}$, any $r\geq 1 $ and any  $t\in\{\text{loc},q\}$, 
 \[\chi_t(\mathcal{G})\leq \chi_t(\mu_{r-1}(\mathcal{G})).\]
\end{proposition}
We thank David~E.~Roberson for pointing out to us the following fact. 
\begin{remark} 
    There exists a classical graph $G$ such that $\chi_{q}(\mu(G))=\chi_{q}(G)$. Indeed, let $G_{13}$ be the graph introduced in \cite{MR2018} and defined as follows. Consider a three-dimensional cube centered on the origin of $\mathbb{R}^3$ and identify vector $v$ with $-v$. The set of vertices $V$ for $G_{13}$ consists of vectors (after the aforementioned identification) representing midpoints of the faces (three of them), midpoints of the edges (six of them), and vertices of the cube (four of them). The pair $\{u,v\}$ forms an edge in $G_{13}$ if and only if $v$ and $u$ are orthogonal as vectors in $\mathbb{R}^3$. Adding a single vertex and connecting it with all the vertices from $G_{13}$ one gets a new graph, called $G_{14}$. By \cite[Lemma~6]{MR2018}, we have $\chi_q(G_{13})=\chi_q(G_{14})$. On the other hand, by construction there is a morphism $\mu(G_{13})\rightarrow G_{14}$, so that $\chi_q(\mu(G_{13}))\leq \chi_q(G_{14})$ by \cite[Proposition~6.4]{BGH}. This shows that $ \chi_{q}(\mu(G_{13}))=\chi_{q}(G_{13})$.
\end{remark}
\section{Clique numbers}
\label{sec:clique}
Let us recall that with an irreflexive graph $G$ with $|V(G)|=n$ we associate an  operator {space} of the form
\[
S_G=\mathrm{span}\{\ket{e_i}\bra{ e_j}  :\, e_i\sim e_j\}\subseteq B(\mathbb{C}^n)
\]
with $\{e_i\}$ being the standard basis of $\mathbb{C}^n$. The clique number $\omega(G)$ of $G$, which is the size of the largest complete graph contained in $G$, can be equivalently defined \cite{Stahlke,Duan_2013} as the maximal cardinality of a set $K$ for which we can find a collection of  non-zero vectors $ \{\psi_k\in\ell^2_n:k\in K\}  $ such that $  \ket{\psi_i}\bra{\psi_j}\in S_G \text{ for } i,j,k\in K \textrm{ and } i\neq j $, i.e.
\begin{equation}
    \omega(G)=\max\left\{|\{\psi_k\in\ell^2_n:k\in K\}| \, :\,\psi_k\neq 0\,, \ket{\psi_i}\bra{\psi_j}\in S_G \text{ for } i,j,k\in K \textrm{ and } i\neq j \right\}.
\end{equation}
This definition can be adapted to the context  of a finite quantum graph $ \mathcal{G}$: the clique number for a quantum graph $\mathcal{G}=(\mathbb{G},\psi,A)$ is given by
\begin{equation}
\label{eq:clique}
    \omega(\mathcal{G})=\max\left\{|\{\psi_k\in L^2(\mathcal{G}):k\in K\}| \, :\,\psi_k\neq 0\,, \ket{\psi_i}\bra{\psi_j}\in S_\mathcal{G} \text{ for } i,j,k\in K \textrm{ and } i\neq j \right\}.
\end{equation}

The clique number can be also defined using the notion of graph homomorphism from a complete graph, i.e., by $    \omega(\mathcal{G})=\max\{|K_n|: \, K_n\rightarrow\mathcal{G}\}$ \cite{Stahlke}.
    
\begin{remark}\label{rem:embclique}
It is easy to check that  the above definition agrees with \eqref{eq:clique}. Moreover, if there is a morphism $\mathcal{G}\to\mathcal{F}$ then $\omega(\mathcal{G})\leq \omega(\mathcal{F})$. In particular $\omega(\mathcal{G})\leq \omega(\mu_{r-1}(\mathcal{G}))$ for every $r\geq 1$. 
\end{remark}

\begin{proposition}\label{clique_thm1}
    Let $\mathcal{G}$ be a quantum graph. Then $\omega(\mu(\mathcal{G}))=\omega(\mathcal{G})$.
\end{proposition}
\begin{proof}
  By the Remark \ref{rem:embclique} we have $\omega(\mathcal{G})\leq \omega(\mu(\mathcal{G}))$.
  
In the proof of the converse inequality let us use the following notation: remembering that as vector spaces we have $L^2(\mu(\mathcal{G})) =\mathbb{C}\oplus  L^2( \mathcal{G} )\oplus L^2( \mathcal{G} )$,  the components of  $\varphi\in L^2(\mu(\mathcal{G}))$ will be denoted by $\varphi^0,\varphi^1,\varphi^2$ respectively. Similarly every $X\in B(L^2(\mu(\mathcal{G})))$ can be written in a matrix form \[X=\begin{pmatrix}
    X_{00}& X_{01}&X_{02}\\
    X_{10}& X_{11}&X_{12}\\
    X_{20}& X_{21}&X_{22}
\end{pmatrix}\] where $X_{00}\in\mathbb{C}$, $X_{ij}\in B(L^2(\mathcal{G}))$ and $X_{i0},X_{0i}^*\in L^2(\mathcal{G})$
 for $i,j\in\{1,2\}$.

Suppose that  the clique number $\omega(\mu(\mathcal{G}))=|K|$ where $K$ is a finite set such that there is a set of non-zero vectors  $ \{\psi_k\in L^2(\mu(\mathcal{G})):k\in K\}$ satisfying $\ket{\psi_i}\bra{\psi_j}\in S_{\mu(\mathcal{G})}$  for $i,j\in K$ and $i\neq j$. 

Let $i\neq j$. Since $|\psi_i\rangle\langle \psi_j|\in S_{\mu(\mathcal{G})}$, there exists $X^{(ij)}\in B(L^2(\mu(\mathcal{G})))$ such that
\begin{equation}\label{mat_str}
    |\psi_i\rangle \langle \psi_j|=\frac{2\delta^2+1}{\delta^2}\begin{pmatrix}
        0 & 0 & m_\bullet (\eta_{\mathcal{G}}^\ast\otimes X_{02}^{(ij)})m_{\mathcal{G}}^\ast \\
        0 & \frac{1}{\delta^2}m_{\mathcal{G}}(A_{\mathcal{G}}\otimes X_{11}^{(ij)})m_{\mathcal{G}}^\ast & \frac{1}{\delta^2}m_{\mathcal{G}}(A_{\mathcal{G}}\otimes X_{12}^{(ij)})m_{\mathcal{G}}^\ast \\
        m_{\mathcal{G}}(\eta_{\mathcal{G}}\otimes X^{(ij)}_{20})m_\bullet^\ast & \frac{1}{\delta^2}m_{\mathcal{G}}(A_{\mathcal{G}}\otimes X^{(ij)}_{21})m_{\mathcal{G}}^\ast &0
    \end{pmatrix} 
\end{equation} where the matrix structure on the right of Equation \eqref{mat_str} is a direct consequence of Definition \ref{def:MT} of $A_{\mu(\mathcal{G})}$.
On the other hand, we have
\begin{equation}\label{eq:lower_corn0}
    |\psi_i\rangle \langle \psi_j|=\begin{pmatrix}
         |\psi_i^0\rangle \langle \psi_j^0| &  |\psi_i^0\rangle \langle \psi_j^1| & |\psi_i^0\rangle \langle \psi_j^2|\\
         |\psi_i^1\rangle \langle \psi_j^0| &  |\psi_i^1\rangle \langle \psi_j^1| & |\psi_i^1\rangle \langle \psi_j^2|\\
         |\psi_i^2\rangle \langle \psi_j^0| &  |\psi_i^2\rangle \langle \psi_j^1| & |\psi_i^2\rangle \langle \psi_j^2|
    \end{pmatrix}= \begin{pmatrix}
         0 &  0 & |\psi_i^0\rangle \langle \psi_j^2|\\
         0 &  |\psi_i^1\rangle \langle \psi_j^1| & |\psi_i^1\rangle \langle \psi_j^2|\\
         |\psi_i^2\rangle \langle \psi_j^0| &  |\psi_i^2\rangle \langle \psi_j^1| & 0
    \end{pmatrix}
\end{equation}
and looking at the right lower corner of the right-hand side of \eqref{eq:lower_corn0} we conclude that   for every $i\neq j$ either $\psi_i^2=0$ or $\psi_j^2 =0 $. Suppose that there is $i_0$ such that $\psi^2_{i_0}\neq 0$ (say without loss of generality  we can take $i_0=1$). Then $\psi^2_{j}=0$ for every $j\neq 1$. Consider now  $i \neq j$ and $i,j\neq 1$. In this case 
\begin{equation}\label{eq:lower_corn}
    |\psi_i\rangle \langle \psi_j|= \begin{pmatrix}
         0 &  0 & 0\\
         0 &  |\psi_i^1\rangle \langle \psi_j^1| &0\\
        0 &  0& 0
    \end{pmatrix}
\end{equation}
and therefore $|\psi_i^1\rangle \langle \psi_j^1| \neq 0$ for all $i \neq j$ and $i,j\neq 1$ and hence $\{\psi_{2}^1,\psi_{2}^1,\psi_{3}^1,\ldots,\psi_{|K|}^1\}$ is a witness of a clique of size $|K|$ in $\mathcal{G}$.  
If $ \psi^2_i  = 0$   for all $i$  then Equation \eqref{eq:lower_corn} holds for all $i,j$. Hence  $\psi^1_i\neq 0$ for all $i$ 
   and therefore $\{\psi_{1}^1,\psi_{2}^1,\psi_{3}^1,\ldots,\psi_{|K|}^1\}$ is a witness of a clique of size $|K|$ in $\mathcal{G}$ and this ends the proof. 
 \end{proof}

{
The techniques used  in the proof of Proposition~\ref{clique_thm1} lead to the following generalization.
\begin{proposition}
    Let $\mathcal{G}$ be a quantum graph and $r\geq 1$. Then $\omega(\mu_{r-1}(\mathcal{G}))=\omega(\mathcal{G})$.
\end{proposition}
\begin{proof}
Suppose that  the clique number $\omega(\mu_{r-1}(\mathcal{G}))=|K|$ where $K$ is a finite set analogous to the one in the proof of Proposition~\ref{clique_thm1}. 

Let $i\neq j$. In a completely similar manner as before, there exists $X^{(ij)}\in B(L^2(\mu_{r-1}(\mathcal{G})))$ such that    
\begin{equation}
\label{eq:mu_psi_psi}
    \begin{split}
        |\psi_i\rangle\langle \psi_j| =\frac{1+r\delta^2}{\delta^2}&\Biggl\{ \iota_0 m_\bullet \left(\eta_{\mathcal{G}}^\ast \otimes X_{0r}^{(ij)}\right)m_{\mathcal{G}}^\ast \iota_r^\ast  + \iota_r m_{\mathcal{G}}\left(\eta_{\mathcal{G}}\otimes X_{r0}^{(ij)}\right)m_\bullet^\ast \iota_0^\ast +\frac{1}{\delta^2}\iota_1 m_{\mathcal{G}}(A_{\mathcal{G}}\otimes X^{(ij)}_{11})m^\ast_{\mathcal{G}}\iota_1^ \ast
        \\
        & + \frac{1}{\delta^2} \sum\limits_{k=1}^{r-1}\left[\iota_k m_{\mathcal{G}}\left(A_{\mathcal{G}}\otimes X_{k\, k+1}^{(ij)}\right)m^\ast_{\mathcal{G}}\iota_{k+1}^\ast + \iota_{k+1} m_{\mathcal{G}}\left(A_{\mathcal{G}}\otimes X^{(ij)}_{k+1\, k}\right)m^\ast_{\mathcal{G}}\iota_k^\ast\right]
        \Biggr\}.
    \end{split}
\end{equation}
Comparing the above expression with $|\psi_i\rangle\langle \psi_j| =\sum\limits_{l,k=0}^{r} \iota_l |\psi_i^l\rangle \langle \psi_j^k|\iota_k^\ast$, in particular by looking at the entry $\iota_r |\psi_i^r\rangle \langle \psi_j^r|\iota_r^\ast$, we conclude that for every $i\neq j$ either $\psi_i^r=0$ or $\psi_j^r =0 $. If there exists $i_0$ such that $\psi^r_{i_0}\neq 0$, then $\psi^r_{j}=0$ for every $j\neq i_0$. For  $i \neq j$ such that $i,j\neq i_0$ we then have
\begin{equation}
    |\psi_i\rangle \langle \psi_j| =\iota_1 |\psi_i^1 \rangle \langle \psi_j^1|\iota_1^\ast + \sum\limits_{k=1}^{r-1}\left(\iota_k |\psi_i^k\rangle \langle \psi_j^{k+1} | \iota_{k+1}^\ast+\iota_{k+1} |\psi_i^{k+1}\rangle \langle \psi_j^{k} | \iota_{k}^\ast\right).
\end{equation}
For $r> 2$, the second term does not vanish automatically. However, looking now at the entry $\iota_{r-1} |\psi_i^r\rangle \langle \psi_j^r|\iota_{r-1}^\ast$ we conclude that for all $i\neq j$ either $\psi^{r-1}_{i}=0$ or $\psi^{r-1}_{j}=0$. Recursively, we show that for all $q=2,\ldots,r$ either $\psi_i^q=0$ or $\psi_j^q=0$. In any of these choices, we can construct a witness of a clique of size $|K|$ in $\mathcal{G}$.
\end{proof}
}

The concept of a complete graph admits a quantum version with a quantum space playing a role of a space of vertices and thus the clique number of a given (quantum graph) admits a quantum version that measures the maximal  size of a complete quantum subgraph.  

\begin{definition}
Given a finite quantum space $\mathbb{G}$ and a $\delta$-form $\psi:\C(\mathbb{G})\to\mathbb{C}$ we define  
a quantum  graph     $(\GG,\psi,A)$ where $A:L^2(\GG,\psi)\to L^2(\GG,\psi)$ is given by 
\[Ax = \delta^2\mathbb{I}\psi(x)-x \] where $\mathbb{I}$ denotes the matrix with all entries equal to $1$.  This quantum graph is denoted $\mathcal{K}_{\GG,\psi}$ and referred to as a complete quantum $(\GG,\psi)$-graph.  
\end{definition}
\begin{remark}
 We shall often skip $\psi$ and instead   writing  $\mathcal{K}_{\GG,\psi}$ we use  $\mathcal{K}_{\GG}$  referring to it as  a complete quantum $ \GG $-graph.
\end{remark}

\begin{definition}
    The quantum clique number for a quantum graph $\mathcal{G}$ is given by 
    \begin{equation}
    \omega_q(\mathcal{G})=\max\{ |\mathcal{K}_{\mathbb{F}}|:\,  \mathcal{K}_{\mathbb{F}}\rightarrow\mathcal{G}\}.
\end{equation}
\end{definition}
{
\begin{remark}Analogously to the Remark \ref{rem:embclique} we see that if there  exists a morphism $\mathcal{G}\rightarrow \mathcal{F}$ then $\omega_q(\mathcal{G})\leq \omega_q(\mathcal{F})$. In particular, if $\mathcal{G}$ is a quantum subgraph of $\mathcal{F}$, then $\omega_q(\mathcal{G})\leq \omega_q(\mathcal{F})$ and hence $\omega_q(\mathcal{G})\leq \omega_q(\mu_{r-1}(\mathcal{G}))$ for every $r\geq 1$.
\end{remark} 
 
\begin{question}
    Is the opposite inequality also true? If not, what is the minimal counterexample? 
\end{question}
}

Next, for a quantum graph $\mathcal{G}$ and a positive semidefinite operator $\Lambda\in B(\mathsf{H})$ we define the operator {space} $S_{\mathcal{G}}\otimes \Lambda\subseteq B(L^2(\mathcal{G})) \otimes B(\mathsf{H})\subseteq B (L^2(\mathcal{G})\otimes \mathsf{H})$. Since $(S_{\mathcal{G}},\C(\mathcal{G}),B(L^2(\mathcal{G})))$ is a quantum graph, we know in particular that $\C(\mathcal{G})'S_{\mathcal{G}}\C(\mathcal{G})'\subseteq S_{\mathcal{G}}$. Defining  $\mathsf{A}=\C(\mathcal{G})\otimes B( \mathsf{H} )$ and noticing that $\mathsf{A}'(S_{\mathcal{G}}\otimes \Lambda)\mathsf{A}'\subseteq S_{\mathcal{G}}\otimes \Lambda$, (where we use $B(\mathsf{H})'=\mathbb{C}$) we conclude  that $(S_{\mathcal{G}}\otimes \Lambda, \mathsf{A}, B (L^2(\mathcal{G})\otimes \mathsf{H}))$ forms a quantum graph which we   denote  by $\mathcal{G}\otimes\Lambda$.

\begin{definition} (\cite[Definition~15]{Stahlke})
    We say that there exists a quantum homomorphism between quantum graphs $\mathcal{G}$ and $\mathcal{F}$, denoted $\mathcal{G}\xrightarrow{\ast}\mathcal{F}$, if there exists a positive semidefinite operator $\Lambda\in B(\mathsf{H})$ s.t. $\mathcal{G}\otimes\Lambda\rightarrow\mathcal{F}$. 
\end{definition}
\begin{remark}
For quantum  homomorphisms, one can define the corresponding clique number and the quantum clique number by simply replacing $\mathcal{G}\rightarrow\mathcal{F}$ by $\mathcal{G}\xrightarrow{\ast}\mathcal{F}$ in their definitions. We denote these parameters by $\omega_{\ast}(\mathcal{G})$ and $\omega_{q\ast}(\mathcal{G})$ respectively.
 In a complete analogy to clique numbers $\omega$ and $\omega_q$ we see that 
    if there exists a  quantum homomorphism $\mathcal{G}\xrightarrow{\ast} \mathcal{F}$   then both $\omega_\ast(\mathcal{G})\leq \omega_\ast(\mathcal{F})$ and $\omega_{q\ast}(\mathcal{G})\leq \omega_{q\ast}(\mathcal{F})$. In particular, if $\mathcal{G}$ is a quantum subgraph of $\mathcal{F}$, then $\omega_\ast(\mathcal{G})\leq \omega_\ast(\mathcal{F})$ and $\omega_{q\ast}(\mathcal{G})\leq \omega_{q\ast}(\mathcal{F})$ and hence $\omega_\ast(\mathcal{G})\leq\omega_\ast(\mu_{r-1}(\mathcal{G}))$ and $\omega_{q\ast}(\mathcal{G})\leq \omega_{q\ast}(\mu_{r-1}(\mathcal{G}))$ for every $r\geq 1$.
 \end{remark}

\section{Outlook and open questions}
\label{sec:open}

We have defined the Mycielski transformation (and its generalized versions) for quantum graphs and demonstrated how it affects their certain parameters, in particular (quantum) clique numbers as well as (quantum) chromatic numbers. In contrast to the classical counterpart,  we were not able to prove that this transformation  automatically enlarges the chromatic number by one. Though we were not able to explicitly construct an example that violates the aforementioned equality, we believe that in general, this equality  is not true.  A similar lack of equality is expected for quantum chromatic numbers.

To a classical graph $G$, one can also associate the Lovász number $\overline{\vartheta}(G)$ \cite{Lovasz79}, which satisfies the monotonicity condition, i.e. the existence of a graph homomorphism $G\rightarrow F$ implies $\overline{\vartheta}(G)\leq \overline{\vartheta}(F)$ \cite[Section~4]{Silva13}. Moreover, $\omega(G)\leq \overline{\vartheta}(G)\leq \chi_{\mathrm{loc}}(G)$. The generalization of the Lovász number into the framework of quantum graphs was proposed in \cite{Duan_2013} (see also \cite[Definition~6]{Stahlke}). By \cite[Theorem~19]{Stahlke}, if for two irreflexive quantum graphs $\mathcal{G}$ and $\mathcal{F}$ we have either $\mathcal{G}\rightarrow \mathcal{F}$ or $\mathcal{G}\xrightarrow{\ast}\mathcal{F}$, then $\overline{\vartheta}(\mathcal{G})\leq \overline{\vartheta}(\mathcal{F})$, so that $\overline{\vartheta}(\mathcal{G})\leq \overline{\vartheta}(\mu_{r-1} (\mathcal{G}))$.

Yet another aspect that we aim to investigate in the forthcoming publication is to study how the proposed Mycielski transformation for quantum graphs affects their (quantum) groups of  symmetries. In addition to the aforementioned open questions, we formulate below further potentially intriguing problems motivated by the known results in classical graph theory.

\subsection{Quantum versions of Motzkin-Straus clique number}
The clique number of a classical graph $G=(V,E)$ can be also computed using the Motzkin-Straus characterization \cite{Motzkin},
\begin{equation}
    1-\frac{1}{\omega(G)}=\max\Bigl\{\langle v, A_G v\rangle: \ v\in \mathbb{R}_+^{|V|}, \quad \sum\limits_{i=1}^{|V|}v_i=1\Bigr\}.
\end{equation}

We now mimic this characterization in the quantum setting. For a given convex closed cone in $\mathcal{S}\subseteq L^2(\mathcal{G})$ we can define the following clique numbers.
\begin{definition}
    The Motzkin-Straus clique number $\omega_{\mathcal{S}}(\mathcal{G})$ for quantum graph $\mathcal{G}$ and convex closed cone $\mathcal{S}\subseteq L^2(\mathcal{G})$ is defined through
    \begin{equation}
        1-\frac{1}{\omega_{\mathcal{S}}(\mathcal{G})}=\max_{v\in \mathcal{S}}\langle v,A_{\mathcal{G}}v\rangle.
    \end{equation}
\end{definition}
\begin{question}
    Characterize (if exist) cones that correspond to the clique numbers for quantum graphs defined in Section~\ref{sec:clique}. Which of the Motzkin-Straus clique numbers are preserved by the Mycielski transformation? 
\end{question}

\subsection{Quantum version of Stiebitz theorem}
For a given classical graph $G$, $n\geq 1$ and $r_j\geq 1$ for $j=1,\ldots, n$, we define 
\begin{equation}
    \mu_{\{r_1,\ldots, r_{n}\}}(G)=\mu_{r_n-1}\left(\ldots \mu_{r_2-1}\left(\mu_{r_1-1}(G)\right)\ldots\right).
\end{equation}
For $n=0$, however, we identify $\{r_1,\ldots,r_n\}$ with $\emptyset$ and put $\mu_\emptyset(G)=G$. For $k\geq 2$ we then define
\begin{equation}
    \mathcal{M}_k=\left\{\mu_{\{r_1,\ldots, r_{k-2}\}}(K_2) \,|\, r_j\geq 1, \, j=1,\ldots,k-2\right\},
\end{equation}
i.e. it is the set of all generalized Mycielski transformations of $K_2$ obtained from $k-2$ consecutive applications of $\mu_{r-1}(\cdot)$ with possibly different $r$s in every iteration, and the following holds:
\begin{theorem} (\cite{Stiebitz}) For every $G\in\mathcal{M}_k$ we have $\chi_{\text{loc}}(G)\geq k$.
\end{theorem}

Let $\mathcal{K}_n$ be the quantum complete graph on $\mathrm{Mat}_n$ equipped with the tracial $\delta$-form $\psi_{n}$, and define
\begin{equation}
    \mathbb{M}_k=\left\{\mu_{\{r_1,\ldots, r_{k-2}\}}(\mathcal{K}_2) \,|\, r_j\geq 1, \, j=1,\ldots,k-2\right\}.
\end{equation}
\begin{question}
    For which type of (quantum) chromatic numbers do we have $\chi_{\bullet}(\mathcal{G})\geq k$ for all $\mathcal{G}\in \mathbb{M}_k$. 
\end{question}

\section*{Acknowledgments}
{ The work of AB was partially funded by the Deutsche Forschungsgemeinschaft (DFG, German Research Foundation) under Germany’s Excellence Strategy – EXC-2111 – 390814868.  We thank David E. Roberson for his helpful comments.}
\bibliography{biblproj}{}

\begin{thebibliography}{10}

\bibitem{Brannan_2019}
Michael Brannan, Alexandru Chirvasitu, Kari Eifler, Samuel Harris, Vern
  Paulsen, Xiaoyu Su, and Mateusz Wasilewski.
\newblock Bigalois extensions and the graph isomorphism game.
\newblock {\em Commun. Math. Phys.}, 375(3):1777--1809, 2020.

\bibitem{BGH}
Michael Brannan, Priyanga Ganesan, and Samuel~J. Harris.
\newblock The quantum-to-classical graph homomorphism game.
\newblock {\em J. Math. Phys.}, 63(11):34, 2022.
\newblock 112204.

\bibitem{QchNb}
Peter~J. Cameron, Ashley Montanaro, Michael~W. Newman, Simone Severini, and
  Andreas Winter.
\newblock On the quantum chromatic number of a graph.
\newblock {\em The Electronic Journal of Combinatorics [electronic only]},
  14(1):R81, 2007.

\bibitem{DawsQGT}
Matthew Daws.
\newblock Quantum graphs: different perspectives, homomorphisms and quantum
  automorphisms.
\newblock {\em arXiv e-prints}, arXiv:2203.08716, 2022.

\bibitem{Silva13}
Marcel~Kenji De~Carli~Silva and Levent Tun{\c{c}}el.
\newblock Optimization problems over unit-distance representations of graphs.
\newblock {\em Electron. J. Comb.}, 20(1):P43, 2013.

\bibitem{Duan_2013}
Runyao Duan, Simone Severini, and Andreas Winter.
\newblock Zero-error communication via quantum channels, noncommutative graphs,
  and a quantum {Lov{\'a}sz} number.
\newblock {\em IEEE Trans. Inf. Theory}, 59(2):1164--1174, 2013.

\bibitem{Erdos}
John~A. Erdos, Aristides Katavolos, and Viktor~S. Shulman.
\newblock Rank one subspaces of bimodules over maximal abelian selfadjoint
  algebras.
\newblock {\em J. Funct. Anal.}, 157(2):554--587, 1998.

\bibitem{Ganesan21}
Priyanga Ganesan.
\newblock Spectral bounds for the quantum chromatic number of quantum graphs.
\newblock {\em Linear Algebra Appl.}, 674:351--376, 2023.

\bibitem{Karp1972}
Richard Karp.
\newblock Reducibility among combinatorial problems (1972).
\newblock In {\em Ideas that created the future. Classic papers of computer
  science}, pages 349--356. Cambridge, MA: MIT Press, 2021.

\bibitem{book}
R.~M.~R. Lewis.
\newblock {\em Guide to graph colouring. {Algorithms} and applications}.
\newblock Texts Comput. Sci. Cham: Springer, 2nd edition edition, 2021.

\bibitem{Lovasz79}
László Lovász.
\newblock On the shannon capacity of a graph.
\newblock {\em IEEE Transactions on Information Theory}, 25(1):1--7, 1979.

\bibitem{MR2018}
Laura Mančinska and David~E. Roberson.
\newblock Oddities of quantum colorings.
\newblock {\em Baltic Journal on Modern Computing}, 4(4):846--859, 2016.

\bibitem{Motzkin}
Theodore~S. Motzkin and Ernst~G. Straus.
\newblock Maxima for graphs and a new proof of a theorem of {Tur{\'a}n}.
\newblock {\em Can. J. Math.}, 17:533--540, 1965.

\bibitem{MRV2018}
Benjamin Musto, David Reutter, and Dominic Verdon.
\newblock A compositional approach to quantum functions.
\newblock {\em J. Math. Phys.}, 59(8):081706, 42, 2018.

\bibitem{Mycielski}
Jan Mycielski.
\newblock On the coloring of graphs. (sur le coloriage des graphs).
\newblock {\em Colloq. Math.}, 3:161--162, 1955.

\bibitem{Stahlke}
Dan Stahlke.
\newblock Quantum zero-error source-channel coding and non-commutative graph
  theory.
\newblock {\em IEEE Trans. Inf. Theory}, 62(1):554--577, 2016.

\bibitem{Stiebitz}
Michael Stiebitz.
\newblock Beitr{\"a}ge zur theorie der f{\"a}rbungskritischen graphen.
\newblock {\em Habilitation Thesis, Technical University Ilmenau}, 1985.

\bibitem{VanNgoc}
Nguyen Van~Ngoc.
\newblock On graph colourings (hungarian).
\newblock {\em PhD thesis, Hungarian Academy of Sciences}, 1987.

\bibitem{weaver10}
Nik Weaver.
\newblock Quantum relations.
\newblock {\em Memoirs of the American Mathematical Society}, 215, 05 2010.

\bibitem{NW2015}
Nik {Weaver}.
\newblock {Quantum graphs as quantum relations}.
\newblock {\em arXiv e-prints}, arXiv:1506.03892, June 2015.

\end{thebibliography}
\bibliographystyle{plain}

\end{document}